\documentclass[11pt,reqno,tbtags,a4paper]{amsart}
\usepackage{amssymb}
\usepackage{url}
\usepackage[square,numbers]{natbib}
\bibpunct[, ]{[}{]}{;}{n}{,}{,}

\title
{A central limit theorem for $m$-dependent variables}

\date{27 August, 2021}

\author{Svante Janson}
\thanks{Partly supported by the Knut and Alice Wallenberg Foundation}
\address{Department of Mathematics, Uppsala University, PO Box 480,
SE-751~06 Uppsala, Sweden}
\email{svante.janson@math.uu.se}
\urladdr{http://www.math.uu.se/svante-janson}

\keywords{$m$-dependent; central limit theorem; asymptotic normality;
Lindeberg condition; Lyapunov condition}
\subjclass[2020]{60F05} 

\numberwithin{equation}{section}

\renewcommand\le{\leqslant}
\renewcommand\ge{\geqslant}

\allowdisplaybreaks

\theoremstyle{plain}
\newtheorem{theorem}{Theorem}[section]
\newtheorem{lemma}[theorem]{Lemma}

\theoremstyle{definition}
\newtheorem{example}[theorem]{Example}

\newtheorem{problem}[theorem]{Problem}
\newtheorem{remark}[theorem]{Remark}

\theoremstyle{remark}

\newenvironment{romenumerate}[1][-10pt]{
\addtolength{\leftmargini}{#1}\begin{enumerate}
 \renewcommand{\labelenumi}{\textup{(\roman{enumi})}}%
 }{\end{enumerate}}

\newcounter{oldenumi}
{\setcounter{oldenumi}{\value{enumi}}
\begin{romenumerate} \setcounter{enumi}{\value{oldenumi}}}
{\end{romenumerate}}

\newcounter{thmenumerate}

\newcounter{xenumerate}   

\newcounter{steps}
\newcommand\stepx[1]{\smallskip\noindent\refstepcounter{steps}%
 \emph{Step \arabic{steps}: #1}\noindent}

\newcommand{\refT}[1]{Theorem~\ref{#1}}
\newcommand{\refTs}[1]{Theorems~\ref{#1}}

\newcommand{\refL}[1]{Lemma~\ref{#1}}

\newcommand{\refR}[1]{Remark~\ref{#1}}

\newcommand{\refS}[1]{Section~\ref{#1}}

\newcommand{\refE}[1]{Example~\ref{#1}}
\newcommand{\refEs}[1]{Examples~\ref{#1}}

\begingroup
  \divide\count255 by 60
  \multiply\count255 by -60
  \advance\count255 by \time
  \ifnum \count255 < 10 \xdef\klockan{\the\count1.0\the\count255}
  \else\xdef\klockan{\the\count1.\the\count255}\fi
\endgroup

\newcommand{\sumin}{\sum_{i=1}^n}

\newcommand\set[1]{\ensuremath{\{#1\}}}

\newcommand\xpar[1]{(#1)}
\newcommand\bigpar[1]{\bigl(#1\bigr)}
\newcommand\Bigpar[1]{\Bigl(#1\Bigr)}

\newcommand\lrpar[1]{\left(#1\right)}
\newcommand\bigsqpar[1]{\bigl[#1\bigr]}
\newcommand\sqpar[1]{[#1]}
\newcommand\xsqpar[1]{[#1]}
\newcommand\Bigsqpar[1]{\Bigl[#1\Bigr]}

\newcommand\lrsqpar[1]{\left[#1\right]}
\newcommand\xcpar[1]{\{#1\}}

\newcommand\Bigcpar[1]{\Bigl\{#1\Bigr\}}

\newcommand\bigabs[1]{\bigl\lvert#1\bigr\rvert}

\newcommand\lrabs[1]{\left\lvert#1\right\rvert}
\def\rompar(#1){\textup(#1\textup)}    
\newcommand\xfrac[2]{#1/#2}

\def\xexp(#1){e^{#1}}
\newcommand\ceil[1]{\lceil#1\rceil}

\newcommand\ntoo{\ensuremath{{n\to\infty}}}

\newcommand\asntoo{\text{as }\ntoo}

\newcommand\bmin{\wedge}
\newcommand\norm[1]{\lVert#1\rVert}

\newcommand\normm[1]{\norm{#1}_2}

\newcommand\punkt{.\spacefactor=1000}    
\newcommand\iid{i.i.d\punkt}    
\newcommand\ie{i.e\punkt}
\newcommand\eg{e.g\punkt}

\newcommand{\tend}{\longrightarrow}
\newcommand\dto{\overset{\mathrm{d}}{\tend}}
\newcommand\pto{\overset{\mathrm{p}}{\tend}}

\newcounter{CC}
\newcommand{\CC}{\stepcounter{CC}\CCx} 
\newcommand{\CCx}{C_{\arabic{CC}}}     
\newcommand{\CCreset}{\setcounter{CC}0} 
\newcounter{cc}

\newcommand\E{\operatorname{\mathbb E{}}}
\renewcommand\P{\operatorname{\mathbb P{}}}

\newcommand\Var{\operatorname{Var}}
\newcommand\Cov{\operatorname{Cov}}

\newcommand\ga{\alpha}

\newcommand\gd{\delta}
\newcommand\gD{\Delta}

\newcommand\gam{\gamma}

\newcommand\gs{\sigma}

\newcommand\gss{\sigma^2}

\newcommand\eps{\varepsilon}

\renewcommand\phi{\xxx}  

\newcommand\cF{\mathcal F}

\newcommand\indic[1]{\boldsymbol1\xcpar{#1}} 
 
\newcommand\Bigindic[1]{\boldsymbol1\Bigcpar{#1}} 

\newcommand\qw{^{-1}}

\newcommand\qq{^{1/2}}
\newcommand\qqw{^{-1/2}}

\newcommand{\gsf}{$\gs$-field}

\newcommand\lhs{left-hand side}
\newcommand\rhs{right-hand side}

\newcommand\xoo{_1^\infty}

\newcommand\mdep{$m$-dependent}
\newcommand\mndep{$(m_n)$-dependent}

\newcommand\gDM{\gD}
\newcommand\nix{_{ni}}
\newcommand\nj{_{nj}}
\newcommand\nk{_{nk}}
\newcommand\sumix[1]{\sum_{#1=i+1}^\Nn}
\newcommand\yy{\upsilon}
\newcommand\sumiln{\sum_{i=1}^{\Nn}}
\newcommand\mn{m_n}
\newcommand\Nnx{N_n}
\newcommand\Nn{{\Nnx}}
\newcommand\Lyap{Lyapunov}

\newcommand\CS{Cauchy--Schwarz}
\newcommand\CSineq{\CS{} inequality}

\begin{document}

\begin{abstract} 
We give a simple and general central limit theorem for a triangular array of
$m$-dependent variables. The result requires only a Lindeberg condition and
avoids unnecessary extra conditions that have been used earlier.
The result applies also to increasing $m=m(n)$, provided the Lindeberg
condition is modified accordingly.
This improves earlier results by several authors.
\end{abstract}

\maketitle

\section{Introduction}\label{S:intro}

Central limit theorems for \mdep{} variables under various conditions
have a long history. 
Pioneering results, for a fixed $m$, were given by
\citet{HoeffdingR} 
and \citet{Diananda} 
(for an \mdep{} sequence),
and
\citet{Orey}
(more generally, and also for a triangular array). 
The results were extended to the case of increasing $m=m(n)$,
see for example
\citet{Bergstrom}, 
\citet{Berk}, 
\citet{RomanoWolf}.

The purpose of the present paper is to give a simple and general 
central limit theorem 
which includes several previous results, but
to our knowledge has not been stated before in this form.
We state first the case of a fixed $m$, where we only have to assume the
usual Lindeberg condition. For notation, see \refS{Snot}.

\begin{theorem}
  \label{Tm}
Let $m\ge0$ be fixed. Suppose that $(X\nix)_{n\ge1,1\le i \le \Nn}$ is an
\mdep{} triangular array
and denote its row sums by $S_n:=\sumiln X\nix$.
Suppose further that the variables $X\nix$
have finite second moments and
$\E X\nix=0$. 
Let
\begin{align}\label{tmgss}
  \gss_n:=\Var S_n 
,\end{align}
and assume that $\gss_n>0$ for all large $n$.
Finally, assume  the usual Lindeberg condition:
for every $\eps>0$, as \ntoo,
\begin{align}\label{tmL}
\frac{1}{\gss_n} \sumiln \E\bigsqpar{X\nix^2\indic{|X\nix|>\eps\gs_n}} \to0.
\end{align}
Then
\begin{align}\label{tmclt}
  S_n/\gs_n\dto N(0,1)
\qquad\asntoo.
\end{align}
\end{theorem}

\begin{remark}
  The case $m=0$ of \refT{Tm}, \ie, an independent array $(X_{ni})$,
is the classical central limit theorem with Lindeberg's condition; see \eg{}
\cite[Theorem XV.6.1 and Problem XV.29]{FellerII},
\cite[Theorem 5.12]{Kallenberg} or \cite[Theorem 7.2.4]{Gut}.  
 Moreover, in this case the Lindeberg condition \eqref{tmL} in necessary
 under a weak extra condition, see
\cite[Theorem XV.6.2]{FellerII} and  \cite[Theorem 5.12]{Kallenberg};
hence we cannot expect a more general theorem for \mdep{} variables 
without this condition (or something stronger).
\end{remark}

\refT{Tm} is only a minor generalization of the result by \citet{Orey},
where the main theorem essentially (ignoring some technical details)
shows the same result under the extra condition
\begin{align}\label{cond+}
  \sumiln\Var X\nix = O\bigpar{\gss_n}.
\end{align}
(See also \cite[Theorem 13.1]{Rosen} which gives another proof of 
Orey's result, now stated similarly to our
\refT{Tm} with the extra condition \eqref{cond+}.)
The condition \eqref{cond+}
is satisfied in most applications, but it is easy to see that
there are cases where \eqref{cond+} does not hold but \refT{Tm} applies,
see \refE{Econd+}.

Note also that this result by \citet{Orey} 
for \mdep{} variables extends to much more mixing conditions.
As shown by \citet[Theorem 2.1]{Peligrad}, \refT{Tm} holds also if we
replace ``\mdep'' by ``strongly mixing'', and add \eqref{cond+} and the
condition $\lim\bar\rho_n^*<1$. (See \cite{Peligrad} for definition, and note
that in the \mdep{} case this is trivial since then $\bar\rho^*_n=0$ when
$n>m$.) 
We will not consider mixing conditions further, but we state a problem.
(There is a large literature on asymptotic normality under various mixing
conditions. See \eg{} 
\cite{Bradley}, which however mainly considers only the
case of stationary sequences, and the references there.)

\begin{problem}
  Does \refT{Tm} extend to suitable mixing conditions?
In particular, does
\cite[Theorem 2.1]{Peligrad} hold also without the assumption \eqref{cond+}?
\end{problem}

More generally, we can allow $m$ to depend on $n$.
(Then mixing results such as
\cite{Peligrad} do not apply. 
However, more complicated results such as \cite{Rio} 
may apply in this case, see \refR{RRio}.)
The statement is almost the same in this case; we only have to
modify the Lindeberg condition.

\begin{theorem}
  \label{Tmn}
Let $(m_n)_n$ be a given sequence of integers with $m_n\ge1$. 
Suppose that $(X\nix)_{n\ge1,1\le i\le\Nn}$ is a \mndep{} triangular array
and denote its row sums by $S_n:=\sumiln X\nix$.
Suppose further that the variables $X\nix$
have finite second moments and
$\E X\nix=0$. 
Let
\begin{align}\label{tmngss}
  \gss_n:=\Var S_n 
,\end{align}
and assume that $\gss_n>0$ for all large $n$.
Finally, assume the following  version of the Lindeberg condition:
for every $\eps>0$, as \ntoo,
\begin{align}\label{tmnL}
\frac{ m_n}{\gss_n} \sumiln
  \E\Bigsqpar{X\nix^2\Bigindic{|X\nix|>\frac{\eps\gs_n}{m_n}}} 
\to0.
\end{align}
Then
\begin{align}\label{tmnclt}
  S_n/\gs_n\dto N(0,1)
\qquad\asntoo.
\end{align}
\end{theorem}

\begin{remark}\label{Rm1}
The assumption $m_n\ge1$ in \refT{Tmn}
is just for convenience. (Otherwise we would have
to replace $m_n$ by $m_n+1$ or $m_n\vee1$ in \eqref{tmnL}.)
It is no real loss of generality, since we may replace any $m_n=0$ by 1.
It is then obvious that \refT{Tm} is a special case of \refT{Tmn}.
\end{remark}

We will see in \refE{EL} that \eqref{tmnL} is the natural version of the
Lindeberg condition when $m$ is allowed to depend on $n$, and that it cannot
be weakened. In particular, \eqref{tmL} is not  enough if $m_n\to\infty$.

As immediate corollaries, the Lindeberg conditions \eqref{tmL} and \eqref{tmnL}
can be replaced by corresponding \Lyap{} conditions; se \refS{SLyap}.
We will also compare this to the results of \cite{Berk} and \cite{RomanoWolf},
and in particular show that their main results follow from \refT{Tmn}
and that some of their conditions are not needed.

\begin{remark}
  There is a large number of papers on various aspects of limits for \mdep{}
  random variables not discussed here.
In particular, we mention results on rate of convergence and Berry--Essen
type estimates, see for example
\cite{Petrov,
Heinrich,
Shergin,
ChenShao-local}.
\end{remark}

\section{Notation}\label{Snot}

We recall some standard notions, and give our notation for them.

Let $m\ge0$ be an integer.
A (finite or infinite) sequence $(X_i)_i$ of random variables is
\emph{\mdep} if the 
two families $\set{X_i}_{i\le k}$ and $\set{X_i}_{i> k+m}$ of random
variables are independent of each other for every $k$. 
In particular, 0-dependent is the same as independent.

A \emph{triangular array} is an array 
of random variables
$\xpar{X\nix}_{n\ge1,\, 1\le i\le \Nn}$, 
for some given sequence $\Nn\ge1$;
it is assumed that the variables $(X\nix)_i$ in a single row are defined on
the same probability space. (No relation is required
between variables in different rows.)

The row lengths $\Nn$ are supposed to be given; we often omit them from
the notation and write \eg{} $\sum_i X\nix$ for the row sum 
$\sum_{i=1}^{\Nn}X\nix$. 

If $m\ge0$ is a fixed integer, we say that the triangular array $(X\nix)$ is
\emph{\mdep} if each row $(X\nix)_i$ is \mdep.
More generally, given a sequence $(m_n)\xoo$ with $m_n\ge0$, 
we say that $(X\nix)$ is \emph{\mndep} if, for every $n\ge1$, 
the row $(X\nix)_i$ is $m_n$-dependent.

For a random variable $X$, $\normm{X}:=\bigpar{\E\sqpar{X^2}}\qq$.

Convergence in probability and distribution is denoted by $\pto$ and $\dto$,
respectively.
Unspecified limits are as \ntoo.

\section{Proof of \refTs{Tm} and \ref{Tmn}}\label{Spf}

We begin with a special case of \refT{Tmn}.
The general case will then follow by a simple truncation argument.

\begin{lemma}\label{Lmn}
    In addition to the assumptions in \refT{Tmn},
assume also that $\gss_n\to1$ as \ntoo, and that $(\eps_n)_n$ is a sequence
with $\eps_n\to0$ such that
\begin{align}\label{lmn0}
  |X\nix|\le \eps_n/m_n 
\qquad\text{a.s.}, 
\end{align}
for  all $n$ and $i$.
Then
\begin{align}\label{lmn2}
  S_n\dto N(0,1)
\qquad\asntoo.
\end{align}
\end{lemma}

\begin{proof}
The idea of the proof is to approximate, for each $n$,  
the sequence of partial sums $\sum_{i=1}^kX_{ni}$
by a martingale $(M_{nk})_{k=0}^\Nn$ with $M_{n0}=0$ and $M_{n\Nn}=S_n$,
see \eqref{tm6} below,
and then use a martingale central limit theorem for $M_{nk}$.
(Note that in the independent case, the sequence of partial sums is a
martingale, but in the \mdep{} case it is in general not; the proof
shows that the martingale \eqref{tm6} is a good approximation.)

The martingale limit theorem that we use is
\cite[Theorem 3.2 with Remarks, pp.~58--59]{HH}, which shows that the
conclusion \eqref{lmn2} follows provided we show that, with 
$\gDM_{nk}:=M_{n,k}-M_{n,k-1}$,
\begin{align}
  \max_k|\gDM_{nk}|&\pto0, \label{HH1}\\
  \sum_k\gDM_{nk}^2&\pto1, \label{HH2}\\
  \E\bigsqpar{\max_k\gDM_{nk}^2}&\le C. \label{HH3}
\end{align}

We separate the proof into several steps.
For notational convenience, we define $X\nix:=0$ for $i\le0$ and $i>\Nn$.

\stepx{The martingale.} 
Let $\cF_{nk}$ be the \gsf{} generated by $X_{n1},\dots,X_{nk}$, and define
\begin{align}
W_{nik}&:=\E\bigpar{X_{ni}\mid \cF_{nk}},\label{tmw}
\\
  M_{nk}&:=\E\bigpar{S_n\mid \cF_{nk}}
=\sum_i W_{nik}
.\label{tm6}
\end{align}
Thus $(M_{nk})_{k=0}^\Nn$ is a martingale for each $n$, with $M_{n0}=\E
S_n=0$,
$M_{n\Nn}=S_n$,
and martingale differences
\begin{align}\label{tm7}
  \gDM_{nk}:=M_{n,k}-M_{n,k-1}
=\sum_i \bigpar{W_{ni,k}-W_{ni,k-1}}.
\end{align}
If $i\le k$, then $X_{ni}$ is $\cF_{nk}$-measurable, and thus
\begin{align}
  \label{WX}
W_{nik}=X_{ni}, 
\qquad i\le k.
\end{align}
In particular, if $i\le k-1$, then
$W_{ni,k}-W_{ni,k-1}=X_{ni}-X_{ni}=0$.
Furthermore, if $i>k+m$, then the $m$-dependence shows that
$X_{ni}$ is independent of $\cF_{nk}$, and thus
\begin{align}
  \label{kab}
W_{nik}=\E\xpar{X_{ni}\mid \cF_{nk}}= \E X_{ni}=0,
\qquad i>k+m.
\end{align}
Hence, \eqref{tm7} simplifies to
\begin{align}\label{az1}
  \gDM_{nk}
=\sum_{i=k}^{k+m}
\bigpar{W_{ni,k}-W_{ni,k-1}}
.\end{align}

Similarly, by \eqref{WX} and \eqref{kab} again,
\begin{align}
  \label{axel}
M_{nk}
=\sum_{i=1}^k X\nix + \sum_{i=k+1}^{k+m}W_{nik}
.\end{align}

We have also, by the martingale property and $M_{n\Nn}=S_n$,
\begin{align}\label{per}
\E\sum_k \gDM_{nk}^2
=
  \sum_k\E \gDM_{nk}^2
=\E M_{n\Nn}^2=   \E S_n^2
.\end{align}

\stepx{Proof of \eqref{HH1} and \eqref{HH3}.}
The assumption \eqref{lmn0} 
and \eqref{tmw} yield
\begin{align}\label{az0}
|W_{nik}|\le \xfrac{\eps_n}{m_n}
\qquad\text{a.s.}
\end{align}
There are $2(m_n+1)$ variables $W$ in the sum in \eqref{az1}, and thus
\eqref{az0} yields
\begin{align}\label{paj2}
  |\gDM_{nk}| \le 2(m_n+1) \frac{\eps_n}{m_n}
\le 4\eps_n
\qquad\text{a.s.}
\end{align}
Since $\eps_n\to0$,
both \eqref{HH1} and \eqref{HH3} follow (trivially) from
\eqref{paj2}.

\stepx{Proof of \eqref{HH2}.}
Let
\begin{align}\label{paj3}
Q_n&:=\sum_k \gDM\nk^2,&
q\nk&:=\E\gDM\nk^2, &
T\nk&:=\sum_{i=1}^k X\nix.  
\end{align}
Then $\E Q_n=\E S_n^2$ by \eqref{per}. Furthermore,
\begin{align}
  \label{paj4}
\Var Q_n &=
\E\lrsqpar{\sumiln \bigpar{\gDM\nix^2-q\nix}}^2
\notag\\&
=
\sumiln\E \bigsqpar{\bigpar{\gDM\nix^2-q\nix}^2}
+
2\sumiln\sumix{j}\E \bigsqpar{ \bigpar{\gDM\nix^2-q\nix}\gDM\nj^2}
.\end{align}
First, using \eqref{paj2} and \eqref{per},
\begin{align}\label{paj5}
  \sum_i\E \bigsqpar{\bigpar{\gDM\nix^2-q\nix}^2}
&=\sum_i\Var\bigsqpar{\gDM\nix^2}
\le\sum_i\E\bigsqpar{\gDM\nix^4}
\le 16\eps_n^2\sum_i\E{\gDM\nix^2}
=16\eps_n^2\E S_n^2
.\end{align}
For the double sum in \eqref{paj4}, we note that since $\gDM\nix$ are
martingale differences, if $i<j<k$, then
$\E \bigsqpar{ \bigpar{\gDM\nix^2-q\nix}\gDM\nj\gDM\nk}=0$; by symmetry, the
same holds if $i<k<j$. Hence,
\begin{align}
  \label{paj6}
\sumiln\sumix{j}\E \bigsqpar{ \bigpar{\gDM\nix^2-q\nix}\gDM\nj^2}
&=
\sumiln\sumix{j}\sumix{k}\E \bigsqpar{ \bigpar{\gDM\nix^2-q\nix}\gDM\nj\gDM\nk}
\notag\\&
=
\sumiln\E\sumix{j}\sumix{k} { \bigpar{\gDM\nix^2-q\nix}\gDM\nj\gDM\nk}
\notag\\&
=\sum_i\E\bigsqpar{ \bigpar{\gDM\nix^2-q\nix}\bigpar{M_{n\Nn}-M\nix}^2}
.\end{align}
Recall that $M_{n\Nn}=S_n=T_{n\Nn}$. The conjugate rule gives
\begin{align}
  \label{paj7}
&\bigpar{M_{n\Nn}-M\nix}^2
=
\bigpar{T_{n\Nn}-M\nix}^2
\notag\\&\qquad
=\bigpar{T_{n\Nn}-T_{n,i+\mn}}^2+\bigpar{2T_{n\Nn}-T_{n,i+\mn}-M\nix}\bigpar{T_{n,i+\mn}-M\nix}
.\end{align}
Hence,
\begin{multline}
  \label{paj8}
\E\bigsqpar{ \bigpar{\gDM\nix^2-q\nix}\bigpar{M_{n\Nn}-M\nix}^2}
=\E\bigsqpar{ \bigpar{\gDM\nix^2-q\nix}\bigpar{T_{n\Nn}-T_{n,i+\mn}}^2}
\\ 
+\E\bigsqpar{\bigpar{\gDM\nix^2-q\nix}
 \bigpar{T_{n,i+\mn}-M\nix}
\bigpar{2T_{n\Nn}-T_{n,i+\mn}-M\nix}}
.\end{multline}
For the first term on the \rhs{} of \eqref{paj8}, we note that $\gD_{ni}$
is $\cF_i$-measurable, and thus the $m$-dependence of $(X\nix)_i$ implies
that $T_{n\Nn}-T_{n,i+\mn}=\sum_{i+\mn+1}^\Nn X_{nk}$ is independent of 
$\gD_{ni}^2-q\nix$. Furthermore, 
$\E\xsqpar{ \gDM\nix^2-q\nix}=0$, and thus
\begin{align}\label{paj9}
\E\bigsqpar{ \bigpar{\gDM\nix^2-q\nix}\bigpar{T_{n\Nn}-T_{n,i+\mn}}^2}
=
\E\bigsqpar{ \gDM\nix^2-q\nix}\E\bigsqpar{\bigpar{T_{n\Nn}-T_{n,i+\mn}}^2}
=0.
\end{align}
Similarly, $T_{n\Nn}-T_{n,i+2\mn}$ is independent  of
$\bigpar{\gDM\nix^2-q\nix} \bigpar{T_{n,i+\mn}-M\nix}$,
and $\E\bigpar{T_{n\Nn}-T_{n,i+2\mn}}=0$;
hence,
\begin{align}\label{paj10}
\E\bigsqpar{\bigpar{\gDM\nix^2-q\nix}
 \bigpar{T_{n,i+\mn}-M\nix}
\bigpar{2T_{n\Nn}-2T_{n,i+2\mn}}}=0.
\end{align}
Consequently, we obtain from \eqref{paj8}--\eqref{paj10}
\begin{multline}
    \label{paj11}
\E\bigsqpar{ \bigpar{\gDM\nix^2-q\nix}\bigpar{M_{n\Nn}-M\nix}^2}
\\ 
=\E\bigsqpar{\bigpar{\gDM\nix^2-q\nix}
 \bigpar{T_{n,i+\mn}-M\nix}
\bigpar{2T_{n,i+2\mn}-T_{n,i+\mn}-M\nix}}
.\end{multline}
The assumption \eqref{lmn0} implies $|T_{n,i+2\mn}-T_{n,i+\mn}|\le \eps_n$,
and also, using \eqref{axel} and \eqref{az0},
\begin{align}\label{paj12}
  |T_{n,i+\mn}-M\nix|
=\lrabs{ \sum_{j=i+1}^{i+\mn}\bigpar{X\nj-W_{nji}}}
\le 2\eps_n.
\end{align}
Hence,
\begin{align}\label{paj12+}
\bigabs{2T_{n,i+2\mn}-T_{n,i+\mn}-M\nix}
\le 2|T_{n,i+2\mn}-T_{n,i+\mn}|+  |T_{n,i+\mn}-M\nix|
\le 4\eps_n, 
\end{align}
and
\eqref{paj11}--\eqref{paj12+} yield
\begin{align}
    \label{paj13}
\E\bigsqpar{ \bigpar{\gDM\nix^2-q\nix}\bigpar{M_{n\Nn}-M\nix}^2}
\le 8\eps_n^2 \E\bigabs{\gDM\nix^2-q\nix}
\le 16\eps_n^2 \E \gDM\nix^2
.\end{align}
Combining \eqref{paj6} 
and \eqref{paj13} yields, using again \eqref{per},
\begin{align}
  \label{paj16}
\sumiln\sumix{j}\E \bigsqpar{ \bigpar{\gDM\nix^2-q\nix}\gDM\nj^2}
&\le
 16\eps_n^2 \sum_i\E \gDM\nix^2
=16\eps_n^2\E S_n^2.
\end{align}
Finally, 
\eqref{paj4}, \eqref{paj5} and \eqref{paj16} 
yield the estimate
\begin{align}
  \label{paj}
\Var \bigsqpar{Q_n}
\le 48\eps_n^2\E S_n^2
=48\eps_n^2\gss_n
\to0,
\end{align}
recalling $\eps_n\to0$ and $\gss_n\to1$.

Consequently,
$Q_n-\E Q_n\pto0$.
Since $\E Q_n=\E S_n=\gss_n\to1$ by \eqref{per} and assumption,
we obtain
\begin{align}
  Q_n\pto1,
\end{align}
which is \eqref{HH2}.
(Recall the definition \eqref{paj3}.)

\stepx{Conclusion.}
We have verified \eqref{HH1}--\eqref{HH3}, and, as said above, 
the asymptotic normality \eqref{lmn2} of $S_n=M_{n\Nn}$ follows by 
\cite[Theorem 3.2 with Remarks, pp.~58--59]{HH}.
\end{proof}

\begin{proof}[Proof of \refT{Tmn}]
First, by replacing $X\nix$ by $X\nix/\gs_n$
(possibly ignoring some small $n$ with $\gs_n=0$), we may and will assume that
$\gs_n=1$ for all $n$.

Next, since \eqref{tmnL} holds for every fixed $\eps>0$, 
it holds also for some sequence $\eps_n\to0$; i.e., 
there exists a sequence $\eps_n\to0$ such that
\begin{align}\label{ma1}
 m_n \sumiln \E\bigsqpar{X\nix^2\indic{|X\nix|>\eps_n/m_n}} \to0.
\end{align}
We fix such a sequence $\eps_n$, and use it to truncate the variables:
 define
\begin{align}\label{sw1}
  X'\nix:=X\nix\indic{|X\nix|\le\eps_n/m_n}-\yy\nix,
&&
  X''\nix:=X\nix\indic{|X\nix|>\eps_n/m_n}+\yy\nix,
\end{align}
where
\begin{align}\label{sw2}
\yy\nix:=\E\bigsqpar{X\nix\indic{|X\nix|\le\eps_n/m_n}}
=-\E\bigsqpar{X\nix\indic{|X\nix|>\eps_n/m_n}}
.\end{align}
Clearly, both $(X'\nix)_{n,i}$ and $(X''\nix)_{n,i}$ are 
triangular arrays with \mndep{} rows and means 0.
Denote the corresponding row sums by $S'_n$ and $S''_n$.

We will estimate $\E(S_n'')^2$ in \eqref{ma4} below; this is an instance of 
an  estimate in \cite[Lemma 13.1]{Rosen}, but
for completeness we include the simple proof.
For any two square-integrable random variables $Y$ and $Z$, we have by the
\CSineq{} and the arithmetic-geometric inequality
\begin{align}\label{ma2}
 | \Cov(Y,Z)|
\le \bigpar{\E\sqpar{Y^2}\E\sqpar{Z^2}}\qq
\le\tfrac12 \bigpar{\E\sqpar{Y^2}+\E\sqpar{Z^2}}.
\end{align}
Hence,  by \eqref{sw1}--\eqref{sw2},
for convenience again defining $X\nix:=0$ for $i\le0$ and $i>\Nn$,
\begin{align}\label{ma4}
  \E\bigsqpar{(S_n'')^2}&
=\sum_{i,j}\Cov\bigpar{X''_{ni},X''_{nj}}
=\sum_{i}\sum_{j=i-\mn}^{i+\mn}\Cov\bigpar{X''_{ni},X''_{nj}}
\notag\\&
\le \sum_{i}\sum_{j=i-\mn}^{i+\mn}\tfrac12
\bigpar{\E \xsqpar{|X''_{ni}|^2}+\E \xsqpar{|X''_{nj}|^2}}
\le(2\mn+1)\sum_i \E \xsqpar{|X''_{ni}|^2}
\notag\\&
\le (2\mn+1) \sum_i\E\bigsqpar{X\nix^2\indic{|X\nix|>\eps_n/m_n}}.
\end{align}
Consequently, \eqref{ma1} implies
\begin{align}\label{ma5}
    \E\bigsqpar{(S_n'')^2} \to0.
\end{align}

In other words, $\normm{S''_n}\to0$, and since
$\normm{S_n}=\gs_n=1$ by assumption
(recalling  $\E S_n=0$), we obtain from
Minkowski's inequality 
$\normm{S_n'}=\normm{S_n-S''_n}\to1$,
and thus 
\begin{align}
  \Var S_n'=\E\bigsqpar{(S_n')^2}=\normm{S_n'}^2\to1.
\end{align}

Furthermore, \eqref{sw1}--\eqref{sw2} imply
$|X'\nix|\le 2\eps_n/m_n$.
Consequently, \refL{Lmn} applies to $(X'\nix)$ (with $2\eps_n$),
which yields
\begin{align}
  S_n'\dto N(0,1).
\end{align}
Since also $S''_n\pto0$ by \eqref{ma5}, 
the conclusion
\eqref{tmnclt} follows by the Cram\'er--Slutsky theorem
\cite[Theorem 5.11.4]{Gut}.
\end{proof}

\begin{proof}[Proof of \refT{Tm}]
\refT{Tm} is, as said in \refR{Rm1}, a special case of \refT{Tmn}.  
\end{proof}

\section{\Lyap{} conditions}\label{SLyap}

It is an immediate corollary of \refTs{Tm} and \ref{Tmn} that
instead of the Lindeberg conditions \eqref{tmL} and \eqref{tmnL}, we may use
a \Lyap{} type condition.
This is often more convenient for applications.
We state such a version of \refT{Tmn}.

\begin{theorem}
  \label{TmnLy}
Suppose that $(X\nix)_{n\ge1,1\le i\le\Nn}$ is a \mndep{} triangular array
with $\E X\nix=0$.
Let $S_n:=\sumiln X\nix$ and $\gss_n:=\Var S_n$, and
assume that $\gss_n>0$ for all large $n$.
Assume also that for some fixed $r>2$, as \ntoo,
\begin{align}\label{lyap}
\frac{ m^{r-1}_n}{\gs_n^r} \sumiln \E{|X\nix|^r} \to0.
\end{align}
Then
\begin{align}\label{tmncltLy}
  S_n/\gs_n\dto N(0,1)
\qquad\asntoo.
\end{align}
\end{theorem}

\begin{proof}
We have
\begin{align}\label{ly1}
\E\Bigsqpar{X\nix^2\Bigindic{|X\nix|>\frac{\eps\gs_n}{m_n}}}
\le \Bigpar{\frac{m_n}{\eps\gs_n}}^{r-2} \E{|X\nix|^r}
\end{align}
and thus
\begin{align}\label{ly2}
\frac{ m_n}{\gss_n} \sumiln 
\E\Bigsqpar{X\nix^2\Bigindic{|X\nix|>\frac{\eps\gs_n}{m_n}}}
\le
\eps^{2-r}
\frac{ m_n^{r-1}}{\gs_n^r} \sumiln 
\E{|X\nix|^r}.
\end{align}
  Hence, \eqref{tmnL} follows from \eqref{lyap}, and \refT{Tmn} applies.
\end{proof}

\begin{remark}
  In the classical case with independent summands, the \Lyap{} condition
  gets stronger as the exponent $r$ increases, so the most general result
  is obtained with $r$ small (i.e., close to 2).
However, this is not always the case here;
see \refEs{Econd+} and  \ref{ERW}.
Thus different values of $r$ yield
incomparable conditions, 
so in an application $r$ may have to be adapted to
the problem.
\end{remark}

We next compare \refT{TmnLy} to the results of \citet{Berk} and
  \citet{RomanoWolf}, and show that their theorems follow from \refT{TmnLy}.
We will see in \refE{ERW} that the implications are strict; there are 
examples where \refT{TmnLy} applies but not \cite{Berk} or \cite{RomanoWolf}.

\begin{example}\label{EBerk}
\citet[Theorem(i)(iii)(iv)]{Berk} assumes, in our notation,
for some $\gd>0$ and constants $C$ and $c$,
\begin{align}
  \E |X\nix|^{2+\gd}&\le C,\label{berki}
\\
\gss_n/\Nn&\to c>0\label{berkiii}
\\
m_n^{2+2/\gd}&=o(\Nn).\label{berkiv}
\end{align}
With $r:=2+\gd$, we obtain  from \eqref{berki}--\eqref{berkiii} (for large $n$)
\begin{align}
  \frac{ m^{r-1}_n}{\gs_n^r} \sumiln \E{|X\nix|^r} 
\le
  \CC\frac{ m^{1+\gd}_n}{\Nnx^{(2+\gd)/2}} \Nn 
=
  \CCx\frac{ m^{1+\gd}_n}{\Nnx^{\gd/2}}
=
 \CCx\lrpar{\frac{ m^{2+2/\gd}_n}{\Nn}}^{\gd/2}.
\end{align}
Hence, \eqref{lyap} follows from \eqref{berkiv}.
Consequently, the theorem in \cite{Berk} is a special case of \refT{TmnLy}.
(Note that we have not used the assumption (ii) in \cite{Berk}; thus 
\refT{TmnLy} is stronger, and more convenient to apply.)
\end{example}

\begin{example}\label{ERomanoWolf}\CCreset
\citet{RomanoWolf} show that their theorem extends the result by
\citet{Berk}  discussed in \refE{EBerk}.
\citet[Theorem 2.1(1)(3)(5)(6)]{RomanoWolf}
assume, in our notation, for some $\gd>0$ and $\gam\in[-1,1)$, 
and some $\gD_n$ and $L_n$,
\begin{align}
  \E|X\nix|^{2+\gd}&\le\gD_n,\label{RW1}
\\
\gss_n/\xpar{\Nn m_n^\gam} &\ge L_n,\label{RW3}
\\
\gD_n /L_n^{(2+\gd)/2}&=O(1),\label{RW5}
\\
m_n^{1+(1-\gam)(1+2/\gd)}/\Nn&\to0.\label{RW6}
\end{align}
With $r:=2+\gd$, we obtain by \eqref{RW1}, \eqref{RW5}, \eqref{RW3}, 
\eqref{RW6}, for some constant $C$,
\begin{align}
    \frac{ m^{r-1}_n}{\gs_n^r} \sumiln \E{|X\nix|^r} 
&\le 
\frac{ m^{1+\gd}_n}{\gs_n^{2+\gd}} \Nn \gD_n
\le 
C\frac{ m^{1+\gd}_n}{\gs_n^{2+\gd}} \Nn L_n^{(2+\gd)/2}
\le 
C\frac{ m^{1+\gd}_n}{(\Nn m_n^\gam)^{(2+\gd)/2}} \Nn 
\notag\\&
=C\frac{ m^{(1+\gd/2)(1-\gam)+\gd/2}_n}{\Nnx^{\gd/2} }
=C\lrpar{\frac{ m^{(2/\gd+1)(1-\gam)+1}_n}{\Nn}}^{\gd/2} 
\to0
.\end{align}
Hence, \eqref{lyap} follows.  
Consequently, the theorem in \cite{RomanoWolf} is a
special case of \refT{TmnLy}. 
(Note that we did not use assumptions (2) and (4) in \cite{RomanoWolf}.
Also, our condition is simpler and seems easier to apply.)
\end{example}

\begin{remark}\label{RRio}
\citet[Corollary 1]{Rio} is a result stated much more generally
for strongly mixing triangular arrays,  where the mixing rate may depend on $n$.
In the special case of an \mndep{} array, we have 
(in the notation of \cite{Rio}) 
$\ga_{(n)}\qw(x)\le m_n+1$, 
and using this it is easy to see that, assuming $m_n\ge1$, condition (b) in the
corollary in \cite{Rio} holds if
\begin{align}
  \label{rio}
\frac{m_n}{\gss_n}
\sum_i\E\Bigsqpar{X\nix^2\Bigpar{\frac{m_n}{\gs_n} |X_n|\bmin1}}
\to0.
\end{align}
Furthermore, it can be seen that \eqref{rio} also implies condition (a) in
the corollary,
and thus the corollary then yields asymptotic normality.

Note that the condition \eqref{rio} is intermediate between \eqref{tmnL} and
\eqref{lyap} for $r=3$. More precisely, it is easily seen that \eqref{rio}
implies \eqref{tmnL} (and thus this special case of \cite[Corollary 1]{Rio}
follows from \refT{Tmn}); 
on the other hand, \eqref{lyap} with $r=3$ implies \eqref{rio}, and thus the
case $r=3$ of \refT{TmnLy} follows from \cite[Corollary~1]{Rio}.

Finally, we note that in \refE{Econd+} below, it follows from 
\eqref{e1gss}, \eqref{marie} and \eqref{e1lyap}
that \eqref{rio} holds only if $\ga>1/3$. Hence, \refTs{Tm} and \ref{Tmn} do
not follow from this special case of \cite{Rio}.
\end{remark}

\section{Examples}
We give some examples illustrating the various conditions.

\begin{example}
  \label{Econd+}
Let $\xi_i$ and $\eta_i$, $i\ge0$, 
be \iid{} random variables with $\P(\xi_i=\pm1)=\P(\eta_i=\pm1)=\frac12$.
Let $\Nn:=n$, let $0<\ga<\frac12$, and define
\begin{align}\label{e1x}
  X_{ni}:= n\qqw \xi_i + n^{-\ga}\bigpar{\eta_i-\eta_{i-1}},
\qquad 1\le i\le n.
\end{align}
Then
\begin{align}\label{e1sn}
  S_n=n\qqw\sumin\xi_i+n^{-\ga}\bigpar{\eta_n-\eta_0}.
\end{align}
It follows that we have
\begin{align}\label{e1gss}
  \gss_n=\Var S_n = 1+2n^{-2\ga}\to1,
\end{align}
and, by the standard central limit theorem,
\begin{align}\label{e1lim}
  S_n\dto N(0,1).
\end{align}
The triangular array $(X\nix)$ is 1-dependent, and \eqref{tmL} is trivial
since
\begin{align}\label{marie}
  |X\nix|\le n\qqw+2n^{-\ga}\to0.
\end{align}
Thus \refT{Tm} applies and yields \eqref{e1lim}.
However,
\begin{align}
  \sumin\Var X\nix 
=n\bigpar{n\qw+2n^{-2\ga}}
=1+ 2n^{1-2\ga}\to\infty,
\end{align}
so \eqref{cond+} does not hold.
Thus, as said in \refS{S:intro}, \refT{Tm} is  more general than
previous versions assuming also \eqref{cond+}.

Furthermore, let us check the \Lyap{} condition \eqref{lyap}. 
We have by \eqref{e1x}, since $n\qqw\ll n^{-\ga}$,
\begin{align}
  \E|X\nix|^r
\sim n^{-r\ga}\E|\eta_i-\eta_{i-1}|^r
= cn^{-r\ga}
\end{align}
for some constant $c>0$. (In fact, $c=2^{r-1}$.)
Hence, 
\begin{align}\label{e1lyap}
\frac{ 1}{\gs_n^r} \sumin \E{|X\nix|^r} 
\sim cn^{1-r\ga}. 
\end{align}
Since $m=1$, \eqref{e1lyap} shows that
\eqref{lyap} holds if $r>1/\ga$, but not if $2<r<1/\ga$.
Hence, in this example, the \Lyap{} condition gets weaker if $r$ is
increased, and not stronger as in the independent case.
\end{example}

\begin{example}\label{EL}
  Let $m_n\ge1$ be a given sequence,
and let $(Y\nix)$ be a triangular array with independent rows.
Define the array $(X\nix)$ by repeating each random variable $Y\nix$ $m_n$
times, and dividing it by $m_n$. In other words, we define
$X\nix:=m_n\qw Y_{n,\ceil{i/m_n}}$.
Then, denoting the row-wise sums by $S_n^X$ and $S_n^Y$, we have
$S_n^X=S_n^Y$.
Moreover, for any $\eps>0$,
\begin{align}
  \sum_i \E\bigsqpar{X\nix^2\indic{|X\nix|>\eps\gs_n/m_n}}
&=
  \sum_jm_n \E\bigsqpar{(Y_{nj}/m_n)^2\indic{|Y_{nj}|>\eps\gs_n}}
\notag\\&
=
m_n\qw  \sum_j \E\bigsqpar{Y_{nj}^2\indic{|Y_{nj}|>\eps\gs_n}}  
.\end{align}
Hence, the condition \eqref{tmnL} is equivalent to the 
usual Lindeberg condition on $(Y_{nj})$.
This shows that \eqref{tmnL} is a natural version of the Lindeberg condition
for \mndep{} arrays, and that it cannot be weakened.

Similarly, the \lhs{} of \eqref{lyap} is the same for $(X\nix)$ and for
$(Y_{nj})$; 
this shows that \eqref{lyap} is a natural version of the \Lyap{} condition
for \mndep{} arrays.
\end{example}

\begin{example}
  \label{ERW}
Let $\xi_i$, $i\ge1$, and $\eta$ be \iid{} $N(0,1)$ variables.
Let $m_n\to\infty$ with $m_n=o(n\qq)$, and take $\Nn:=n+m_n\sim n$.
Define
\begin{align}
  X\nix:=
  \begin{cases}
    \xi_i, & 1\le i\le n,\\
\eta, & n<i\le n+m_n.
  \end{cases}
\end{align}
Then $(X\nix)$ is an   \mndep{} triangular array. Furthermore,
\begin{align}
  \gss_n=n+m_n^2\sim n.
\end{align}
Moreover, $S_n\in N(0,\gss_n)$, so \eqref{tmclt} is trivial.
The \lhs{} of \eqref{lyap} is
\begin{align}
  \sim \frac{m_n^{r-1}}{n^{r/2}}\Nn \E|\xi_1|^r
\sim c_r\frac{m_n^{r-1}}{n^{r/2-1}}
\end{align}
for some constant $c_r>0$, and thus \eqref{lyap} holds if and only if
\begin{align}\label{elmn}
  m_n = o\bigpar{n^{(r-2)/(2(r-1))}}.
\end{align}
Note that the exponent in \eqref{elmn} increases with $r$.
Consequently,
as in \refE{Econd+} but for another reason,
the \Lyap{} condition \eqref{lyap}  gets weaker if $r$ is increased.
In the present example, choosing a larger $r$ means weakening the
restriction on $m_n$.
On the other hand, if we modify the example and let $\xi_i$ and $\eta$ have
some other (centred) distribution, a larger $r$ also means a stronger moment
condition on the variables, so there might be a trade-off.

We note also that this example does not satisfy the conditions in
\cite{RomanoWolf}, and thus not the stronger conditions in \cite{Berk}.
To see this, note that conditions (2), (4) and (3) in \cite{RomanoWolf}
(choosing $k=m_n$ and $a=n+1$) imply
\begin{align}
m_n^2
=  \Var \sum_{i=n+1}^{m_m+n} X_{ni}
\le C m_n^{1+\gam} \frac{\gss_n}{\Nn m_n^\gam}
\sim Cm_n
\end{align}
and thus $m_n=O(1)$, contradicting our assumptions.
\end{example}

\newcommand\AAP{\emph{Adv. Appl. Probab.} }
\newcommand\JAP{\emph{J. Appl. Probab.} }
\newcommand\JAMS{\emph{J. \AMS} }
\newcommand\MAMS{\emph{Memoirs \AMS} }
\newcommand\PAMS{\emph{Proc. \AMS} }
\newcommand\TAMS{\emph{Trans. \AMS} }
\newcommand\AnnMS{\emph{Ann. Math. Statist.} }
\newcommand\AnnPr{\emph{Ann. Probab.} }
\newcommand\CPC{\emph{Combin. Probab. Comput.} }
\newcommand\JMAA{\emph{J. Math. Anal. Appl.} }
\newcommand\RSA{\emph{Random Structures Algorithms} }
\newcommand\DMTCS{\jour{Discr. Math. Theor. Comput. Sci.} }

\newcommand\AMS{Amer. Math. Soc.}
\newcommand\Springer{Springer-Verlag}
\newcommand\Wiley{Wiley}

\newcommand\vol{\textbf}
\newcommand\jour{\emph}
\newcommand\book{\emph}
\newcommand\inbook{\emph}
\def\no#1#2,{\unskip#2, no. #1,} 
\newcommand\toappear{\unskip, to appear}

\newcommand\arxiv[1]{\texttt{arXiv:#1}}
\newcommand\arXiv{\arxiv}

\newcommand\xand{and }
\renewcommand\xand{\& }

\def\nobibitem#1\par{}

\end{document}